\documentclass{amsart}

\usepackage{tikz-cd}

\usepackage{amssymb}

\usepackage{graphicx}
\usepackage{xcolor}                            
\usepackage[justification=justified]{caption}  

\usepackage{hyperref}

\numberwithin{equation}{section}
\numberwithin{figure}{section}

\newcommand{\norm}[1]{\lVert #1 \rVert}
\newcommand{\abs}[1]{\lvert #1 \rvert}

\newcommand{\diam}{\operatorname{diam}}
\newcommand{\Hold}{\operatorname{H\ddot{o}ld}}

\newcommand{\RR}{\mathbb{R}}
\newcommand{\ZZ}{\mathbb{Z}}
\newcommand{\XX}{\mathbb{X}}
\newcommand{\YY}{\mathbb{Y}}

\newtheorem{theorem}{Theorem}[section]
\newtheorem{prop}[theorem]{Proposition}
\newtheorem{lemma}[theorem]{Lemma}
\newtheorem{cor}[theorem]{Corollary}

\theoremstyle{remark}
\newtheorem{question}[theorem]{Question}

\title[Space-filling surfaces]{Space-filling surfaces: sharp H\"older continuous parameterizations from squares to cubes}
\date{September 18, 2026}
\author{Matthew Badger}
\author{Kevin Palmer}
\subjclass[2020]{Primary 26B35, 28A80; Secondary 51F30, 54C05, 11A63}
\keywords{space-filling surface, sharp H\"older exponent, Arnold's problem, Peano curve, self-similar set,
$b$-adic representation}
\address{Department of Mathematics\\ University of Connecticut\\ Storrs, CT 06269-1009}
\email{matthew.badger@uconn.edu}
\address{Department of Mathematics\\ University of Connecticut\\ Storrs, CT 06269-1009}
\email{kevin.2.palmer@uconn.edu}
\begin{document}

\begin{abstract} Following a hint of Semmes, we employ Stong's bijections between integer lattices to construct space-filling surfaces, which are higher-dimensional analogues of space-filling curves. For each $m\geq 2$ we build $\alpha$-H\"older continuous parameterizations $f:[0,1]^m\rightarrow[0,1]^{m+1}$ with sharp exponent $\alpha=m/(m+1)$. In particular, there exist $(2/3)$-H\"older continuous surjections from squares to cubes. This solves Arnold's problem 1988--5.
\end{abstract}

\maketitle

\section{Introduction} \label{sec:intro}

Can a disk be deformed continuously into a ball? How about a square into a cube? To understand the difficulty, imagine crumpling a sheet of paper without tearing or stretching it. The resulting map $f$ from the paper rectangle into $\RR^{3}$ is Lipschitz continuous with constant 1, i.e., $\norm{f(x)-f(y)}_2\leq\norm{x-y}_2$ for all $x$ and $y$,
because folding brings pairs of points closer together and never farther apart. The more you compress the sheet, twisting and contorting it, the more the wad of paper resembles the 3d region it occupies. In the limit, though, it is crushed to a point. You might hope to recover a 3d shape by stretching the sheet back out, but dilating at every scale uniformly costs us equicontinuity, and with it any guarantee of a continuous limit. Thus, passing from a 2d shape to a 3d one requires an alternative approach.

Using a powerful topological tool, one can see that there exist continuous surjections from squares to cubes. Call a space a \emph{curve} if it is the image of a continuous map of a compact interval. The Hahn--Mazurkiewicz theorem, e.g., see \cite{sagan}, says that a metrizable space is a curve precisely when it is compact, connected, and locally connected. The cube has all three of these properties, so it is a curve and there is a continuous onto map $h:[0,1]\rightarrow [0,1]^3$. By composing $h$ with a projection $\pi:[0,1]^2\rightarrow[0,1]$, we obtain a continuous surjection $f=h\circ \pi:[0,1]^{2} \rightarrow [0,1]^{3}$. This settles the question of existence and nothing more. The argument offers no control on the modulus of continuity of $f$, and no way to judge whether $f$ deforms the square into the cube optimally.

One way to measure the distortion of a map is through the H\"older condition. Given an \emph{exponent} $0 \leq \alpha \leq 1$, we say that a map $f$ between metric spaces is \emph{$\alpha$-H\"older} if there exists a constant $0\leq H<\infty$ such that \begin{equation}\label{Holder-condition}
\norm{f(x)-f(y)} \leq H \norm{x-y}^{\alpha}\quad\text{for all $x$ and $y$}.\end{equation} Let $\Hold_\alpha(f)\in[0,\infty]$ denote the infimal $H$ such that \eqref{Holder-condition} holds, noting that $f$ is $\alpha$-H\"older if and only if $\Hold_\alpha(f)<\infty$. It is easy to see that when $\Hold_\alpha(f)<\infty$, the inequality \eqref{Holder-condition} holds with $H=\Hold_\alpha(f)$. Moreover, a map is $0$-H\"older if and only if it has bounded image, and a map that is $\alpha$-H\"older with positive exponent is continuous. A $1$-H\"older map is usually called \emph{Lipschitz}. If $0\leq \alpha\leq \beta\leq 1$, then any $\beta$-H\"older map on a bounded domain is $\alpha$-H\"older. Larger exponents $\alpha$ in the H\"older condition allow a map to distort distances less severely at small scales. Thus, the supremum $\alpha^*$ of exponents $\alpha$ for which there exists an $\alpha$-H\"older map from a bounded metric space $\XX$ onto a metric space $\YY$ serves as a gauge of least possible distortion of a continuous map from $\XX$ onto $\YY$. Once the \emph{critical H\"older exponent} $\alpha^*=\alpha^*(\XX,\YY)$ is identified, we can also ask whether or not it is attained.

Suppose that $1\leq m\leq n$. Throughout the paper we equip $\RR^m$ and $\RR^n$ with the $\ell_\infty$ norm unless stated otherwise. The orthogonal projection $\pi:[0,1]^n\rightarrow[0,1]^m$ from a higher-dimensional cube onto a lower-dimensional cube is Lipschitz with constant 1. In the other direction, suppose that $f:[0,1]^m\rightarrow[0,1]^n$ is an $\alpha$-H\"older surjection from a lower-dimensional cube onto a higher-dimensional cube. Partition the unit cube $[0,1]^{m}$ into $k^{m}$ cubes of side length $1/k$. By the H\"older condition, the image $f(Q_i)$ of each cube $Q_i$ from the partition has $\ell_\infty$ diameter at most $\Hold_{\alpha}(f)k^{-\alpha}$ and $[0,1]^n\subset \bigcup_{i=1}^{k^m} f(Q_i)$. Comparing volumes gives \begin{equation}\label{vol-est} 1=\operatorname{vol}([0,1]^n)\leq \text{total volume of the images $f(Q_i)$}\leq \Hold_\alpha(f)^n k^{m-\alpha n}\end{equation} for every $k\geq 1$. A moment's reflection shows that this is impossible as $k\rightarrow\infty$ unless $\alpha \leq m/n$. Thus, the critical H\"older exponent $\alpha^*([0,1]^m,[0,1]^n)\leq m/n$.

Constructions of space-filling curves $f:[0,1]\rightarrow[0,1]^n$ that are $(1/n)$-H\"older continuous are now well-known. The first examples, with $n=2$, appeared at the end of the 19th century and are due to Peano \cite{peano1890} and Hilbert \cite{hilbert1891}. In 1988, Arnold asked what is the critical exponent of a H\"older map from a square onto a cube and whether the critical exponent is attainable. See \cite[Problem 1988--5]{Arn04}. Shchepin \cite{shc10} solved the first part of Arnold's problem by proving that there exist $\alpha$-H\"older surjections $f_\alpha:[0,1]^m\rightarrow[0,1]^n$ for every $\alpha<m/n$. Thus, in view of the previous paragraph, the critical H\"older exponent from an $m$-cube to an $n$-cube is $m/n$. In this note, we solve the second part of Arnold's problem by showing the critical exponent can be attained.

\begin{theorem}\label{thm:main} For all $m\geq 2$, there is a surjection $f:[0,1]^m\rightarrow[0,1]^{m+1}$ satisfying \begin{equation}\label{the-param}\|f(x)-f(y)\|_\infty \leq 2^{m+2}(2^m-1) \|x-y\|^{m/(m+1)}_\infty\quad\text{for all }x,y\in[0,1]^m.\end{equation} In particular, there exist $(2/3)$-H\"older continuous maps from squares onto cubes.\end{theorem}

Composing parameterizations of space-filling curves and space-filling surfaces between consecutive dimensions immediately gives the following:

\begin{cor}\label{cor:global} For all $1\leq m\leq n$, there exist $(m/n)$-H\"older continuous maps from $[0,1]^m$ onto $[0,1]^n$.\end{cor}

The proof of Theorem \ref{thm:main} is split between Sections~\ref{sec:selfsimilar} and \ref{sec:stong}. We present additional variants of Theorem \ref{thm:main} and discuss the problem of finding best H\"older constants in \eqref{the-param} in Section~\ref{sec:further}.

To build a space-filling curve, one typically breaks the domain interval $[0,1]$ and target cube $[0,1]^n$ into an equal number of subcubes, labeled so that consecutive intervals are assigned to neighboring subcubes. This assignment is then refined in a self-similar fashion. Shchepin~\cite[Theorem~4]{shc10} proved that no map of this kind exists from the $m$-cube onto the $n$-cube when $m<n<2m$, which explains why classic Peano and Hilbert constructions of space-filling curves fail to produce space-filling surfaces. Other approaches to space-filling surfaces with different merits were given by Ahmed and Bokhari~\cite{ahmed-bokhari} and Paulsen~\cite{paulsen}, but do not achieve the critical H\"older exponent.

Shchepin has announced a solution of Arnold's problem \cite{shc26} in a brief
note that states theorems and formulas but includes no proofs. A few weeks
after the first version of this paper appeared on arXiv, Shchepin posted a manuscript on
ResearchGate \cite{shc26-cross} that proves the theorem by a different argument and
complements the proof given below. We obtained our results independently, with
complete proofs and explicit H\"older constants.

The starting point for our construction of space-filling surfaces is Stong's solution to a discrete variant of Arnold's problem.

\begin{theorem}[Stong~\cite{stong1998}]\label{thm:stong-holder}
    For all integers $1 \leq m \leq n$, there is an $(m/n)$-H\"older bijection from $\ZZ^{m}$ onto $\ZZ^{n}$.
\end{theorem}

In an unpublished manuscript \cite[Chapter 9]{Semmes03}, Semmes notes that existence of an $(m/n)$-H\"older surjection $f:\RR^m\rightarrow\RR^n$ should follow by using an appropriate compactness theorem to take a limit of rescaled copies of Stong's map on finer and finer lattices.
To prove Theorem \ref{thm:main}, we follow the spirit of Semmes' suggestion, but with an essential improvement. We give a direct generalization of Stong's argument that avoids normal families. The original idea of Stong is to take an $m$-fold product of an $m/(m+1)$-H\"older map $g$ from $\ZZ$ onto a fractal-like $(m+1)/m$-dimensional set $X_m\subset\ZZ^2$. To obtain Theorem \ref{thm:stong-holder}, Stong then post-composes $\bigotimes_1^m g:\ZZ^m\rightarrow X_m^m$ with a linear transformation from $X_m^m$ onto $\ZZ^{m+1}$. We design continuous analogues of $X_m$ and $g$ in Section~\ref{sec:selfsimilar} and use Stong's transformation to complete the proof of Theorem \ref{thm:main} in Section~\ref{sec:stong}. Key technical steps are the H\"older estimate Proposition \ref{thm:curve} and the bi-Lipschitz estimate Proposition \ref{prop:bilip}. Up to a harmless affine change of coordinates, the final map $f:[0,1]^m\rightarrow[0,1]^{m+1}$ is then a composition \begin{equation}\label{eq:factor-f}
\begin{tikzcd}[column sep=huge]
{[0,1]^m} \arrow[r, "\otimes_1^m g"] \arrow[rrr, bend right=10, "f"']
  & E_\infty^m \arrow[r, "L"]
  & L(E_\infty^m) \arrow[r, "\pi"]
  & {[0,1]^{m+1}},
\end{tikzcd}
\end{equation} where $E_\infty\subset\RR^2$ is a connected self-similar set of Hausdorff dimension $(m+1)/m$ presented as an $m/(m+1)$-H\"older curve $E_\infty=g([0,1])$, $L:\RR^{2m}\rightarrow\RR^{m+1}$ is a rank $m+1$ linear transformation, and $\pi:\RR^{m+1}\rightarrow[0,1]^{m+1}$ is a projection.

Beyond cubes, the construction of H\"older and Lipschitz parameterizations is an active subject. Balka and Keleti~\cite{BK24} proved a very general existence theorem for nearly sharp H\"older surjections. Badger and Schul~\cite{BS25} use square packings to give a sharper criterion for existence of Lipschitz surjections with Euclidean domains and metric targets. An application of Theorem \ref{thm:main} to the relationship between Federer integral rectifiability \cite{Fed47} and Mart\'in-Mattila fractional rectifiability \cite{MM93} appears in Badger and Vellis~\cite{BV19}.

\section{Parameterization of X-fractals with small H\"older constant} \label{sec:selfsimilar}

Throughout Sections~\ref{sec:selfsimilar} and~\ref{sec:stong}, we fix an integer $m \geq 2$ and assign $M=2^m$ and $\alpha=m/(m+1)$. The case of interest for Arnold's square-to-cube problem is $m = 2$, $M = 4$, and $\alpha = 2/3$. Recall we equip vector spaces with the $\ell_\infty$-norm.

\begin{figure}[t]
    \centering
    \includegraphics[width=.95\textwidth]{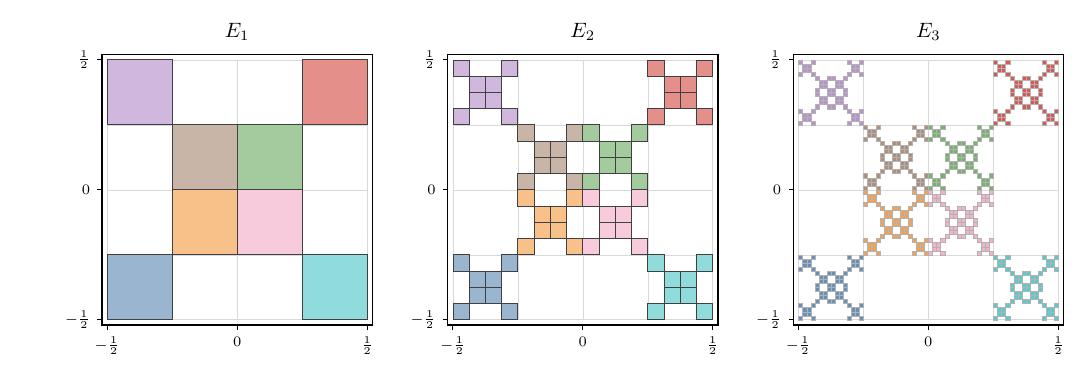}
    \caption{From left to right, we illustrate the first three iterates $E_1$, $E_2$, and $E_3$ of $Q=[-\tfrac12, \tfrac12]^{2}$ under the IFS $\mathcal{E}$ when $m=2$.}
    \label{fig:IFS-Iterations}
\end{figure}

We now define an ``X-shaped'' self-similar set $E_\infty$ in $\RR^2$ with similarity dimension $1/\alpha$; see Figure \ref{fig:IFS-Iterations}. Let $Q := \bigl[-\tfrac12, \tfrac12\bigr]^{2}$. Divide $Q$ into $M^{2}$ closed subsquares $Q_{p,q}$ of side length $1/M$, indexed by $(p,q)\in \{0, \ldots, M-1\}^2$. We keep the $2M$ subsquares that intersect one of the two diagonals of $Q$ and throw out the rest, i.e., we keep $Q_{p,q}$ provided that $(p,q)$ belongs to \begin{equation}\label{alphabet}\mathcal{A}:=\{(p,q)\in\{0,\ldots,M-1\}^2:p=q\text{ or }p+q=M-1\}.\end{equation} Let $\phi_{p,q}$ denote the unique affine map formed using only dilations and translations that sends $Q$ to $Q_{p,q}$, i.e.~
        \begin{equation}
            \phi_{p,q}(x) := \tfrac{1}{M} x + c_{p,q}, \qquad
            c_{p,q} := \bigl( -\tfrac12 + \tfrac{2p+1}{2M},\ -\tfrac12 + \tfrac{2q+1}{2M} \bigr).
        \end{equation}
    Then $\mathcal{E}=\{\phi_{p,q}:(p,q)\in\mathcal{A}\}$ is a connected, self-similar iterated function system (IFS) that satisfies the open set condition (OSC) and has similarity dimension $\log_M(2M)=(m+1)/m$; see e.g.~\cite[Section 2]{BV21}.
    For a word $w = i_1 \cdots i_n$ in the alphabet $\mathcal{A}$, we write $\phi_w$ for the composition $\phi_{i_1}\circ\cdots\circ\phi_{i_n}$. For all words $w$ of length $|w|=n$, the image $\phi_w(Q)$ is a square of side length $M^{-n}$. Define $E_k := \bigcup_{\abs w = k} \phi_w(Q)$. Then the intersection $E_\infty := \bigcap_{k \geq 0} E_k$ is a nonempty compact set and $E_\infty = \bigcup_{\phi \in \mathcal{E}} \phi(E_\infty)$. That is, $E_\infty$ is the attractor of $\mathcal{E}$, and by Hutchinson's theorem \cite{Hut81}, $E_\infty$ has Hausdorff dimension $\log_{M}(2M)=1/\alpha$. Since each $E_k$ is connected, Hata's theorem \cite{Hat85} ensures that $E_\infty$ is connected.

    Remes' theorem \cite{Remes98} says that every connected self-similar set in Euclidean space with the OSC is a $\beta$-H\"older curve, where $1/\beta$ is the similarity dimension. For an exposition and some extension of Remes' theorem, see Badger and Vellis \cite{BV21}. In particular, it follows that $E_\infty$ is an $\alpha$-H\"older curve. Unfortunately, the H\"older constant in Remes' theorem is implicit and is quite large. To obtain a parameterization of $E_\infty$ with a small H\"older constant, we implement the graph-directed IFS method introduced by Rao and Zhang~\cite{RZ16} and further developed with Dai in \cite{DRZ19} and \cite{RZ20}. The proof of Proposition \ref{thm:curve} is self-contained.

\begin{prop}\label{thm:curve}
    There is an $\alpha$-H\"older surjection $g \colon [0,1] \to E_\infty$, $\Hold_\alpha(g) \leq 4$.
\end{prop}

We first construct an $\alpha$-H\"older parameterization $s$ of the attractor of an auxiliary IFS of triangles. Rotated copies of $s$ are assembled into $g$ at the end of the section.

Every line segment below is horizontal or vertical. For each such segment $[a,b]$, oriented from $a$ to $b$, let $\Delta[a,b]$ denote the isosceles right triangle with hypotenuse $[a,b]$ whose opposite vertex lies to its left. Then $\diam \Delta[a,b] = \norm{b-a}_\infty$.

Let $q_1 := (0,0)$ and $q_2 := (1,0)$, and let $T := \Delta[q_1,q_2]$, whose third vertex is $(\tfrac12,\tfrac12)$. Put $K := M/2$ and $p_0 := q_1$, and for $0 \leq j < M$ let
    \[
        \begin{aligned}
            p_{2j+1} &:= \Bigl( \tfrac{j+1}{M}, \tfrac{j}{M} \Bigr), &
            p_{2j+2} &:= \Bigl( \tfrac{j+1}{M}, \tfrac{j+1}{M} \Bigr) & &\text{for all }j < K, \\
            p_{2j+1} &:= \Bigl( \tfrac{j}{M}, \tfrac{M-1-j}{M} \Bigr), &
            p_{2j+2} &:= \Bigl( \tfrac{j+1}{M}, \tfrac{M-1-j}{M} \Bigr) & &\text{for all } j \geq K.
        \end{aligned}
    \]
Consecutive points differ by $1/M$ in exactly one coordinate. The first $M$ of them climb from $q_1$ to $p_M$, the vertex of $T$ opposite the hypotenuse, and the last $M$ descend to $p_{2M} = q_2$, tracing the staircase on the left of Figure~\ref{fig:substitution}. They are distinct except that $p_{M-1} = p_{M+1}$. The last triangle of the climb and the first of the descent therefore share a hypotenuse and lie on opposite sides of it.

For $1 \leq i \leq 2M$, let $\psi_i$ be the orientation-preserving similarity carrying $q_1$ to $p_{i-1}$ and $q_2$ to $p_i$. Its ratio is $1/M$ and $\psi_i(T) = \Delta[p_{i-1},p_i] \subset T$. Then $\mathcal{F} := \{\psi_1, \ldots, \psi_{2M}\}$ is a connected, self-similar IFS satisfying the OSC, with similarity dimension $\log_M(2M) = (m+1)/m$. For a word $w = i_1 \cdots i_k$ in the alphabet $\{1, \ldots, 2M\}$, we write $\psi_w$ for the composite map $\psi_{i_1}\circ\cdots\circ\psi_{i_k}$ and call $\psi_w([q_1,q_2])$ a \emph{level-$k$ segment} and $\psi_w(T)$ its triangle, of diameter $M^{-k}$. Taken in lexicographic order, the $(2M)^{k}$ level-$k$ segments meet head to tail and run from $q_1$ to $q_2$. Define $V_k := \bigcup_{\abs w = k} \psi_w(T)$. Then the intersection $V_\infty := \bigcap_{k \geq 0} V_k$ is a nonempty compact set and $V_\infty = \bigcup_{\psi \in \mathcal{F}} \psi(V_\infty)$. That is, $V_\infty$ is the attractor of $\mathcal{F}$.

Divide $[0,1]$ into intervals $I_{k,1}, \ldots, I_{k,(2M)^{k}}$ of equal length, write $J_k := I_{1,k} = [\ell_k, r_k]$, and put $\Delta_k := \psi_k(T)$.
\begin{figure}[t]
    \centering
    \includegraphics[width=\textwidth]{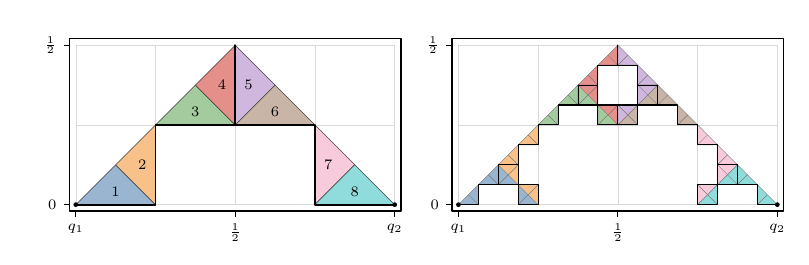}
    \caption{First two levels of the triangle system when $m = 2$. The unions of the hypotenuses are the images of $s_1$ and $s_2$.}
    \label{fig:substitution}
\end{figure}

\begin{lemma}\label{lem:param}
    There is a continuous surjection $s \colon [0,1] \rightarrow V_\infty$ with $s(0) = q_1$ and $s(1) = q_2$ such that, whenever $w$ is the $j$-th word of length $k$ (in the sense above),
        \begin{equation}\label{eq:self}
            s(x) = \psi_w \bigl( s ((2M)^{k}x - j + 1) \bigr) \quad \text{for all } x \in I_{k,j}
        \end{equation}
    and
        \begin{equation}\label{eq:localize}
            s(I_{k,j}) \subset \psi_w(T) .
        \end{equation}
\end{lemma}

\begin{proof}
    Let $s_k \colon [0,1] \rightarrow T$ be the piecewise affine constant speed map that traverses the $j$-th level-$k$ segment on $I_{k,j}$, with $s_k(0) = q_1$ and $s_k(1) = q_2$. See Figure~\ref{fig:substitution}.

    Let $w$ be the $j$-th word of length $k$. The level-$(k+l)$ segments descending from the $j$-th level-$k$ segment are the images under $\psi_w$ of the level-$l$ segments, in the same order. Hence $s_{k+l}(x) = \psi_w ( s_l ((2M)^{k}x - j + 1) )$ for $x \in I_{k,j}$.

    Now let $l \geq k$. The set $s_l(I_{k,j})$ lies in the triangle $\psi_w(T)$, and so does $s_k(I_{k,j})$, which is the hypotenuse of $\psi_w(T)$. Since $\psi_w(T)$ has $\ell_\infty$ diameter $M^{-k}$, it follows that $\norm{s_l - s_k}_\infty \leq M^{-k}$. Hence the maps $s_k$ converge uniformly to some continuous function $s$ that runs from $q_1$ to $q_2$. Letting $l \rightarrow \infty$ gives \eqref{eq:self} and \eqref{eq:localize}.

    By \eqref{eq:localize}, the image $\Gamma=s([0,1])$ lies in every $V_k$, hence $\Gamma\subset V_\infty$. Conversely, any $z\in V_\infty$ lies in a level-$k$ triangle for each $k$. Thus, $z$ is within $M^{-k}$ of an endpoint of that triangle's hypotenuse, and such endpoints belong to $\Gamma$. As $\Gamma$ is compact, we see that $z\in\Gamma$ and $\Gamma=V_\infty$.
\end{proof}

Fix $s$ as in Lemma~\ref{lem:param} and write $s = (\xi,\eta)$ for its coordinate functions. Then $s$ maps the interval $J_k$ into the triangle $\Delta_k$ by \eqref{eq:localize}. Moreover, since $(2M)^{\alpha} = M$, the ratio $\norm{s(x)-s(y)}_\infty / \abs{x-y}^{\alpha}$ is unchanged when an interval $I_{k,j}$ is carried onto $[0,1]$ by \eqref{eq:self}. Because $s$ takes values in $T = \{(x,y) : 0 \leq y \leq \min(x,\, 1-x)\}$,
    \begin{equation}\label{eq:in-T}
        0 \leq \eta \leq \min\{\xi,\ 1-\xi\} .
    \end{equation}
It follows that $\norm{s(x)-q_1}_\infty = \xi(x)$ and $\norm{q_2 - s(x)}_\infty = 1 - \xi(x)$ for every $x$.

\begin{lemma}\label{lem:gen}
    For all $u\in [0,1]$, we have $\xi(1-u) = 1 - \xi(u)$, $\eta(1-u) = \eta(u)$, and
        \begin{equation}\label{eq:gen}
            \left.\begin{aligned}
                s\bigl( \tfrac{2j+u}{2M} \bigr) &= \tfrac{1}{M} \bigl( j + \xi(u),\ j + \eta(u) \bigr) \\
                s\bigl( \tfrac{2j+1+u}{2M} \bigr) &= \tfrac{1}{M} \bigl( j+1-\eta(u),\ j + \xi(u) \bigr)
            \end{aligned}\;\right\}\quad \text{for all }0 \leq j < K,
        \end{equation}
        \begin{equation}\label{eq:gen-down}
            \left.\begin{aligned}
                s\bigl( \tfrac{2j+u}{2M} \bigr) &= \tfrac{1}{M} \bigl( j + \eta(u),\ M-j-\xi(u) \bigr) \\
                s\bigl( \tfrac{2j+1+u}{2M} \bigr) &= \tfrac{1}{M} \bigl( j + \xi(u),\ M-1-j+\eta(u) \bigr)
            \end{aligned}\;\right\}\quad \text{for all }K \leq j < M.
        \end{equation}
\end{lemma}

\begin{proof}
    The four displayed identities are \eqref{eq:self} with $k=1$. The segments in question run from $p_{2j}$ to $p_{2j+1}$ and from $p_{2j+1}$ to $p_{2j+2}$, and they are traversed to the right and upward when $j < K$, downward and to the right when $j \geq K$. The corresponding similarity $\psi_i$ is therefore $M^{-1}$ times the identity, the quarter-turn, or its inverse, followed by the translation carrying $q_1$ to the initial point of the segment. The first assertion holds because reflection in the line $x = \tfrac12$ carries the system of triangles to itself and reverses the order of the segments. The same identities hold with $s_{k+1}$ and $s_k$ in place of $s$.
\end{proof}

\begin{lemma}\label{lem:ends}
    Let $\lambda := 2^{1-\alpha} = 2^{1/(m+1)}$. For all $x \in [0,1]$,
        \begin{equation}\label{eq:Lambda}
            \norm{s(x)-q_1}_\infty \leq \lambda x^{\alpha}
            \quad\text{and}\quad
            \norm{q_2-s(x)}_\infty \leq \lambda (1-x)^{\alpha} .
        \end{equation}
\end{lemma}

\begin{proof}
    Put $\widehat\xi(x) := \min\bigl\{ x + \tfrac{1}{2M},\ \tfrac{x+1}{2} \bigr\}$ for $0 \leq x \leq 1$; see Figure~\ref{fig:xi}. It is increasing and concave, with $\widehat\xi(1) = 1$, $\widehat\xi \leq 1$ and $2\widehat\xi(x) \leq x+1$. We claim that $\xi \leq \widehat\xi$.

    \begin{figure}[t]
    \centering
    \includegraphics[width=.60\textwidth]{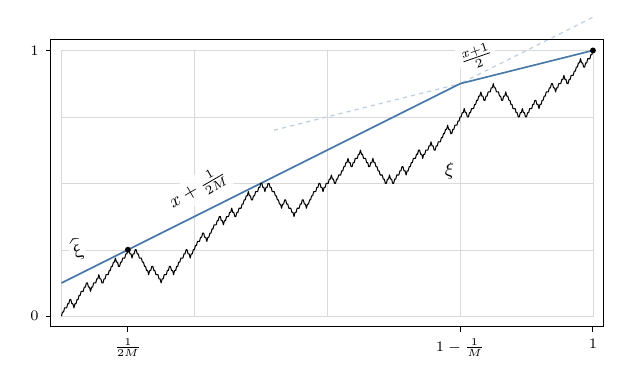}
    \caption{Graphs of $\xi$ (the first coordinate of $s$, drawn using its level 3 approximation) and $\widehat\xi$ when $m = 2$.}
    \label{fig:xi}
\end{figure}

    It is enough to prove the claim for the approximations $s_k = (\xi_k, \eta_k)$ of Lemma~\ref{lem:param} and let $k \rightarrow \infty$. Each $s_k$ takes values in $T$, and \eqref{eq:in-T} holds for $\xi_k$ and $\eta_k$. In particular $\eta_k \leq \tfrac12$. For $k = 0$ it holds because $\xi_0(x) = x \leq \widehat\xi(x)$. Suppose it holds for $s_k$ and write $x = (i+u)/(2M)$ with $0 \leq i < 2M$ and $u \in [0,1]$. Substituting the inductive hypothesis into \eqref{eq:gen} and \eqref{eq:gen-down} leaves four cases.
    \begin{itemize}\setlength{\itemsep}{0pt}\setlength{\parskip}{0pt}\setlength{\topsep}{2pt}
        \item $i = 2j$ with $j < K$. Then $\xi_{k+1}(x) \leq (j + \widehat\xi(u))/M$. Since $x \leq \frac{2j+1}{2M} \leq \frac{M-1}{2M}$, we have $\widehat\xi(x) = x + \frac{1}{2M} = \frac{2j+u+1}{2M}$. The required inequality is $2\widehat\xi(u) \leq u+1$.
        \item $i = 2j+1$ with $j < K$. Then $\xi_{k+1}(x) \leq (j+1)/M$ by \eqref{eq:in-T}, while $\widehat\xi$ is increasing and $\widehat\xi\bigl( \frac{2j+1}{2M} \bigr) = \frac{2j+2}{2M} = \frac{j+1}{M}$.
        \item $i = 2j$ with $j \geq K$. Then $\xi_{k+1}(x) \leq (j+\tfrac12)/M = \frac{2j+1}{2M}$, while $\widehat\xi(x) \geq \widehat\xi(j/M) = \min\bigl\{ \frac{2j+1}{2M},\ \frac{j+M}{2M} \bigr\}$, and $j \leq M-1$.
        \item $i = 2j+1$ with $j \geq K$. Then $\xi_{k+1}(x) \leq (j+\widehat\xi(u))/M$, and the two branches of $\widehat\xi(x)$ call for $2\widehat\xi(u) \leq u+2$ and $4\widehat\xi(u) \leq 2M+1+u-2j$. The first holds because $\widehat\xi \leq 1$, and the second because $2\widehat\xi(u) \leq u+1$ and $j \leq M-1$.
    \end{itemize}

    Now let $0 < x \leq 1$. Taking $j = 0$ in \eqref{eq:gen} gives $\xi(x) = \xi(2Mx)/M$ for $x \leq \frac{1}{2M}$. The ratio $\xi(x)/x^{\alpha}$ is therefore unchanged when $x$ is replaced by $2Mx$, and we may assume that $x \geq \frac{1}{2M}$. On each of the two affine pieces of $\widehat\xi$ the ratio $\widehat\xi(x)/x^{\alpha}$ is largest at an endpoint, because the derivative of $(c+kx)x^{-\alpha}$ has the sign of $(1-\alpha)kx - \alpha c$. At $x = \frac{1}{2M}$ and at $x = 1$ that ratio equals $1$, and at the corner $x = 1 - \frac1M$ it equals $\bigl( 1 - \tfrac{1}{2M} \bigr)\bigl( 1 - \tfrac1M \bigr)^{-\alpha} \leq \tfrac{2M-1}{2M-2} = 1 + \tfrac{1}{2M-2} \leq 1 + \tfrac{1}{2(m+1)} \leq 2^{1/(m+1)} = \lambda$, using $(1-\frac1M)^{\alpha} \geq 1 - \frac1M$, then $2M - 2 \geq 2(m+1)$, then $2^{t} \geq 1 + t\ln 2$ with $t = \frac{1}{m+1}$ and $\ln 2 \geq \frac12$. We conclude that $\xi(x) \leq \lambda x^{\alpha}$, the first bound in \eqref{eq:Lambda}. The second follows from the first and Lemma~\ref{lem:gen}, since $1 - \xi(x) = \xi(1-x)$.
\end{proof}

\begin{lemma}\label{lem:holder}
    The map $s$ is $\alpha$-H\"older with $\Hold_\alpha(s) \leq 4^{1/(m+1)}$. In particular, $V_\infty$ is an $\alpha$-H\"older curve.
\end{lemma}

\begin{proof}
    Let $x < y$ in $[0,1]$. The intervals of level $n$ have length $(2M)^{-n}$, and there is a largest $n$ for which some interval of level $n$ contains both $x$ and $y$. Carrying that interval onto $[0,1]$ leaves the ratio $\norm{s(x)-s(y)}_\infty/ \abs{x-y}^{\alpha}$ unchanged. We may therefore assume that $n = 0$. Then $x \in J_k$ and $y \in J_{k'}$ with $k < k'$. Let $\nu := k'-k-1$ be the number of intervals of level one between them, put $a := 2M(r_k - x)$ and $b := 2M(y - \ell_{k'})$, both in $[0,1]$, and put $r := a+b$. Then $\abs{x-y} = (\nu+r)/(2M)$, and since $(2M)^{\alpha} = M$ the ratio to be bounded is $M\norm{s(x)-s(y)}_\infty/(\nu+r)^{\alpha}$. Carrying $J_k$ and $J_{k'}$ onto $[0,1]$ and applying \eqref{eq:Lambda} at both ends gives
        \begin{equation}\label{eq:ends}
            M\norm{s(x)-s(r_k)}_\infty \leq \lambda a^{\alpha}
            \quad\text{and}\quad
            M\norm{s(\ell_{k'})-s(y)}_\infty \leq \lambda b^{\alpha} .
        \end{equation}
    If $\nu = 0$, then $r_k = \ell_{k'}$ and the triangle inequality with \eqref{eq:ends} gives
        \[
            M\norm{s(x)-s(y)}_\infty \leq M\norm{s(x)-s(r_k)}_\infty + M\norm{s(\ell_{k'})-s(y)}_\infty \leq \lambda \bigl( a^{\alpha} + b^{\alpha} \bigr).
        \]
    By H\"older's inequality with conjugate exponents $m+1$ and $1/\alpha$,
        \begin{equation}\label{eq:conc}
            \lambda \bigl( a^{\alpha} + b^{\alpha} \bigr) = \lambda(1,1)\cdot(a^\alpha,b^\alpha) \leq \lambda^{2} (a+b)^{\alpha}
            \quad \text{for all } a, b \geq 0 .
        \end{equation}
    Since $a+b = r$, the ratio is at most $\lambda^{2}$.

    Suppose next that $\nu \geq 1$. Let $D := \norm{s(r_k)-s(\ell_{k'})}_\infty$ and $\delta := \diam(\Delta_k \cup \Delta_{k'})$. Estimating through $s(r_k)$ and $s(\ell_{k'})$ by \eqref{eq:ends} and \eqref{eq:conc}, or else that $s(x) \in \Delta_k$ and $s(y) \in \Delta_{k'}$, gives
        \begin{equation}\label{eq:sep}
            M \norm{s(x)-s(y)}_\infty \leq \min\bigl\{ \lambda^{2} r^{\alpha} + DM,\ \delta M \bigr\} .
        \end{equation}
Consecutive points $p_i$ lie $1/M$ apart, and the steps alternate between
horizontal and vertical apart from the two vertical steps meeting at $p_M$.
The horizontal steps all point right and the vertical ones point up and then
down. Hence any $n$ consecutive steps span at most $\lceil n/2 \rceil/M$ in
either coordinate, and $DM \leq \lceil \nu/2 \rceil$. The set $\Delta_k \cup
\Delta_{k'}$ is spanned by the $\nu+2$ steps from $p_{k-1}$ to $p_{k'}$
together with the two vertices opposite their hypotenuses, each lying
$1/(2M)$ beyond its own step in the orthogonal direction. A coordinate
receiving $j$ of those vertices is orthogonal to $j$ of the extreme steps,
and the other $\nu+2-j$ steps span at most $(\nu+3-j)/(2M)$ along it. Thus,
$\delta M \leq (\nu+3)/2$.

    Let $\rho := 2^{-2/m}$. Then $\lambda^{2}\rho^{\alpha} = 1$, and we show that
        \begin{equation}\label{eq:key}
            \lambda^{2}(\nu+\rho)^{\alpha} \geq \tfrac{\nu+3}{2} \qquad \text{for all }1 \leq \nu \leq 2M-2.
        \end{equation}
    Its two sides are a concave and an affine function of $\nu$, and only the two endpoints need be checked. At $\nu = 1$ the claim is that $\lambda^{2}(1+2^{-2/m})^{\alpha} \geq 2$, that is, $(1+2^{-2/m})^{\alpha} \geq 2^{(m-1)/(m+1)}$; raising both sides to the power $1/\alpha$ turns this into $1 + 2^{-2/m} \geq 2 \cdot 2^{-1/m}$, which is $(1 - 2^{-1/m})^{2} \geq 0$. At $\nu = 2M-2$ it suffices to drop $\rho$ and prove $2\lambda^{2}(2M-2)^{\alpha} \geq 2M+1$. Here $(2M-2)^{\alpha} = M\bigl(1-\frac1M\bigr)^{\alpha} \geq M - 1$. It therefore suffices that $2\lambda^{2}(M-1) \geq 2M+1$, that is, that $\lambda^{2} \geq 1 + \frac{3}{2M-2}$. For $m = 2$ the two sides are $\lambda^{2} = 4^{1/3}$ and $\tfrac32$. For $m \geq 3$, $2M-2 \geq 3(m+1)$ gives $\frac{3}{2M-2} \leq \frac{1}{m+1} \leq \lambda^{2}-1$.

    Both entries of \eqref{eq:sep} are now handled at once. If $r \leq \rho$, we use the first entry: the quantity $\lambda^{2}\bigl[ (\nu+r)^{\alpha}-r^{\alpha} \bigr]$ decreases in $r$ and is at least $\lambda^{2}(\nu+\rho)^{\alpha} - 1$, which by \eqref{eq:key} is at least $\frac{\nu+1}{2} \geq \lceil \nu/2 \rceil$. If $r \geq \rho$, we use the second entry, which is at most $(\nu+3)/2 \leq \lambda^{2}(\nu+\rho)^{\alpha} \leq \lambda^{2}(\nu+r)^{\alpha}$. Either way the ratio is at most $\lambda^{2} = 4^{1/(m+1)}$.
\end{proof}

\begin{proof}[Proof of Proposition~\ref{thm:curve}]
    Since a translation changes neither H\"older constants nor surjectivity, we may use coordinates in which $Q = [0,1]^{2}$. Listed counter-clockwise from the bottom left, the corners of $Q$ are then $q_1$, $q_2$, $q_3 := (1,1)$, and $q_4 := (0,1)$. These are fixed points of $\phi_{0,0}$, $\phi_{M-1,0}$, $\phi_{M-1,M-1}$, and $\phi_{0,M-1}$. Hence
        \begin{equation}\label{eq:corners}
            q_1,\, q_2,\, q_3,\, q_4 \in E_\infty .
        \end{equation}
    Let $R$ denote the counter-clockwise rotation by $90^{\circ}$ about the center of $Q$. Then $R(q_1)=q_2$, $R(q_2)=q_3$, $R(q_3)=q_4$, and $R(q_4)=q_1$. Let $g$ be the concatenation of the curves $s$, $R\circ s$, $R^2\circ s$, and $R^3\circ s$ of Lemma~\ref{lem:param}, each traversed on one of four intervals of equal length. Figure~\ref{fig:stages} shows three levels of approximation to $g$ when $m = 2$.

    \begin{figure}[t]
    \centering
    \includegraphics[width=\textwidth]{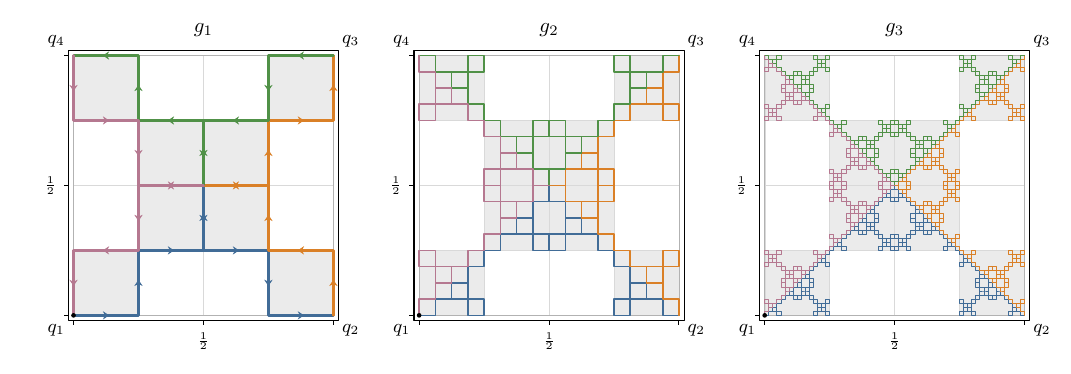}
    \caption{The images of $s_1$, $s_2$, and $s_3$ together with their three rotated copies when $m = 2$: $s$ occupies the lower triangle $T$, $Rs$ the right, $R^2s$ the upper, and $R^3s$ the left.}
    \label{fig:stages}
\end{figure}

    Let $x < y$ in $[0,1]$. Each $R^{j}$ is an isometry and each of the four intervals has length $\tfrac14$. An estimate for $s$ transfers to one of them at the cost of a factor $4^{\alpha}$. If $x$ and $y$ lie in a common interval, Lemma~\ref{lem:holder} bounds the ratio by $4^{\alpha} \cdot 4^{1/(m+1)}$. If they lie in adjacent intervals meeting at $\tau$, then \eqref{eq:Lambda} bounds $\norm{g(x)-g(y)}_\infty$ by $4^{\alpha}\lambda\bigl( (\tau-x)^{\alpha} + (y-\tau)^{\alpha} \bigr)$, and \eqref{eq:conc} with $a = \tau-x$ and $b = y-\tau$ gives $4^{\alpha} \cdot 4^{1/(m+1)}$. Otherwise $\abs{x-y} \geq \tfrac14$ while $\norm{g(x)-g(y)}_\infty \leq \diam Q = 1$, and the ratio is at most $4^{\alpha}$. Therefore,
$\Hold_{\alpha}(g) \leq 4^{\alpha} \cdot 4^{1/(m+1)} = 4^{(m+1)/(m+1)} = 4$.

    It remains to identify the image of $g$. The triangles $\Delta_{2j+1}$ and $\Delta_{2j+2}$ lie in $Q_{j,j}$ when $j < K$ and in $Q_{j,M-1-j}$ when $j \geq K$, and those indices belong to $\mathcal{A}$. The diagonals of such a square cut it into four triangles, each congruent to $M^{-1}T$, and the $8M$ triangles obtained in this way are exactly the $8M$ triangles $R^{j}\Delta_k$. Let $W := V_\infty \cup RV_\infty \cup R^{2}V_\infty \cup R^{3}V_\infty$, the image of $g$. We want to show that $W = E_\infty$. Since $V_\infty$ is the attractor of $\mathcal{F}$, $W = \bigcup_{i \in \mathcal{A}} \phi_i(W)$. Since an IFS has a unique attractor by Hutchinson's theorem, $W = E_\infty$.
\end{proof}

\section{Stong's transformation and the space-filling surface} \label{sec:stong}

\subsection{\texorpdfstring{$M$}{M}-adic representations} \label{subsec:base-M}

Recall $M=2^m$. The \emph{$M$-adic representation} of an integer $v$, called the generalized base-$M$ expansion in~\cite{stong1998}, is the infinite string $\cdots d_2\, d_1\, d_0$ with digits $d_i \in \{0, 1, \ldots, M-1\}$ obtained as follows. If $v \geq 0$, take the ordinary base-$M$ expansion of $v$ and pad it with infinitely many leading zeros. If $v < 0$, take the ordinary expansion of $-v-1$ and replace each digit $d$ by its complement $M-1-d$. The digits are eventually all $0$ when $v \geq 0$ and eventually all $M-1$ when $v < 0$. We call this eventual digit the \emph{stable digit} of $v$.

Since only the stable digit repeats infinitely often, an $M$-adic representation is determined by finitely many digits, and we write
        \[
            v = \bigl[ \overline{d_p},\, d_{p-1},\, \ldots,\, d_0 \bigr]_M
        \]
for the integer whose $M$-adic representation is $d_{p-1}, \ldots, d_0$ preceded by infinitely many copies of $d_p$. The overline marks the stable digit, which lies in $\{0, M-1\}$. It is convenient to let $d_i := d_p$ for every $i \geq p$ as well. Then $d_i$ is defined for all $i \geq 0$ and names the $i$-th digit of $v$ whether or not it is written. The two extreme cases are $0 = [\overline{0}]_M$ and $-1 = [\overline{M-1}]_M$, where no digit appears outside the overline. We will use repeatedly that complementing every digit of the expansion of $v$ produces the expansion of $-v-1$.

It may help to see the expansion in the smallest case. Take $m = 2$. Then $M = 4$ and each digit lies in $\{0,1,2,3\}$. Since $22 = 1 \cdot 16 + 1 \cdot 4 + 2$ and $5 = 1 \cdot 4 + 1$,
$22 = [ \overline{0},\, 1,\, 1,\, 2 ]_4$ and $-6 = [ \overline{3},\, 2,\, 2 ]_4$,
the digits of $-6$ being the complements $d \mapsto 3-d$ of those of $5 = -(-6)-1$, the stable digit included.

\subsection{Stong's set and bijection} \label{subsec:Xm-F}

Following \cite{stong1998}, set
\begin{equation}\label{Xm-def} X_m := \bigl\{ (a,b) \in \ZZ^{2} : d_i = c_i \text{ or } d_i = M - 1 - c_i \text{ for every } i \geq 0 \bigr\}\end{equation}
where $a = [\overline{d_p},\, d_{p-1},\, \ldots,\, d_0]_M$ and $b = [\overline{c_q},\, c_{q-1},\, \ldots,\, c_0]_M$ are the $M$-adic representations of $a$ and $b$. We call this the \emph{digit condition}. Equivalently, $(a,b) \in X_m$ precisely when $(d_i, c_i)$ belongs to the alphabet $\mathcal{A}$ of \eqref{alphabet} for every $i \geq 0$, and Lemma~\ref{lem:gamma} shows that this agreement is no coincidence.

For $k \geq 0$, put
$h_k := M^{k}/2$ and $S_k := X_m \cap [-h_k, h_k)^{2}$.
Then $S_k$ is the part of $X_m$ inside a square of side $M^{k}$ centered at
the origin, and $S_0 = \{(0,0)\}$. Fix one of the $M^{k}$ integers $b \in
[-h_k, h_k)$. Any $a$ with $(a,b) \in S_k$ is determined by $d_0, \ldots,
d_{k-1}$ and its stable digit. The digit condition allows two values for
each, and exactly one of the two stable digits puts $a$ in $[-h_k, h_k)$.
Hence $\# S_k = M^{k} \cdot 2^{k} = 2^{(m+1)k}$.

Motivated by Stong~\cite{stong1998}, consider the linear map $L \colon \RR^{2m} \to \RR^{m+1}$ defined by
\begin{equation}\label{eq:L}
L\bigl((a_1,b_1), \ldots, (a_m,b_m)\bigr) := \Bigl( \textstyle\sum_{j=1}^{m} 2^{j-1} a_j,\ b_1, \ldots, b_m \Bigr).
\end{equation}
We write $P := \sum_{j=1}^{m} 2^{j-1} a_j$ for the first coordinate of an image point. Stong's map is the restriction $F := L|_{X_m^{m}}$. The digit condition plays no part in the definition of $L$. Its role is to make $F$ a bijection, and in Section~\ref{subsec:surface} we apply $L$ to points that are not lattice points.

\begin{prop}[Stong~\cite{stong1998}]\label{prop:bij}
    The map $F \colon X_m^{m} \to \ZZ^{m+1}$ is a bijection.
\end{prop}

Since $L$ is a linear map between finite-dimensional Banach spaces, $L$ and $F = L|_{X_m^{m}}$ are Lipschitz. We now prove that $F$ is bi-Lipschitz.

\begin{prop}\label{prop:bilip}
    For all $z \in \RR^{2m}$ and all $x \in X_m^{m}$,
$\norm{L(z)}_\infty \leq (M-1)\norm{z}_\infty$ and $M^{-1}\norm{x}_\infty \leq \norm{F(x)}_\infty$. Both constants are sharp.
\end{prop}

\begin{proof}
    \emph{Upper bound.} The rows of $L$ have coefficient sums $\sum_{j=1}^{m} 2^{j-1} = M-1$ and $1$, of which the first is the larger. Therefore, $\norm{L(z)}_\infty \leq (M-1)\norm{z}_\infty$.

    \emph{Lower bound.} Fix a point $x = \bigl((a_1,b_1), \ldots, (a_m,b_m)\bigr)$ of $X_m^{m}$, write $F(x) = (P, b_1, \ldots, b_m)$, and set $N := \norm{F(x)}_\infty$. If $N = 0$, then $x = 0$ by Proposition~\ref{prop:bij}. Hence we may assume $N \geq 1$. Write $d_{j,i}$, $c_{j,i}$, and $P_i$ for the $i$-th $M$-adic digits of $a_j$, $b_j$, and $P$. We now set up the equation $\sum_{j=1}^{m} 2^{j-1} a_j = P$ as a grade school addition in base $M$, listing each term $a_j$ with multiplicity $2^{j-1}$:
    \[
        \renewcommand{\arraystretch}{1.0}
        \setlength{\arraycolsep}{4pt}
        \begin{array}{r|cccccccc}
            \text{position} & \cdots & i+1 & i & i-1 & \cdots & 2 & 1 & 0 \\
            \hline
            \text{carry}    &        & \sigma_{i} & \sigma_{i-1} & \sigma_{i-2} &  & \sigma_{1} & \sigma_{0} &  \\
            a_1 & \cdots & d_{1,i+1} & d_{1,i} & d_{1,i-1} & \cdots & d_{1,2} & d_{1,1} & d_{1,0} \\
            a_2 & \cdots & d_{2,i+1} & d_{2,i} & d_{2,i-1} & \cdots & d_{2,2} & d_{2,1} & d_{2,0} \\
            a_2 & \cdots & d_{2,i+1} & d_{2,i} & d_{2,i-1} & \cdots & d_{2,2} & d_{2,1} & d_{2,0} \\
            \vdots & & \vdots & \vdots & \vdots &  & \vdots & \vdots & \vdots \\
            a_m & \cdots & d_{m,i+1} & d_{m,i} & d_{m,i-1} & \cdots & d_{m,2} & d_{m,1} & d_{m,0} \\
            \hline
            P & \cdots & P_{i+1} & P_i & P_{i-1} & \cdots & P_{2} & P_{1} & P_{0}
        \end{array}
    \]
    Since $\sum_{j=1}^{m} 2^{j-1} = M-1$, the list has $M-1$ terms. We perform the addition one position at a time, beginning at position $0$, where there is no carry. The entries in position $i$, together with the carry $\sigma_{i-1}$ into that position, total $\sigma_{i-1} + \sum_{j=1}^{m} 2^{j-1} d_{j,i}$. The last $M$-adic digit of this total is recorded in position $i$ of the sum, and the multiples of $M$ that remain are carried into position $i+1$. Writing $\sigma_i$ for the number carried, and formally setting $\sigma_{-1} := 0$, we obtain
        \begin{equation}\label{eq:carry}
            \sigma_{i-1} + \textstyle\sum_{j=1}^{m} 2^{j-1} d_{j,i} \ =\ P_i + M \sigma_i \quad \text{for all }i \geq 0.
        \end{equation}
    Each of the $M-1$ entries in a position is at most $M-1$, and $(M-2) + (M-1)^{2} < M(M-1)$. Induction on $i$ now gives
        \begin{equation}\label{eq:carry-bound}
            0 \leq \sigma_i \leq M-2 .
        \end{equation}

    Put $t := \lfloor \log_M N \rfloor + 1$. Since $\abs P$ and every $\abs{b_j}$ are at most $N < M^{t}$, the digits $P_i$ and $c_{j,i}$ with $i \geq t$ are stable digits, hence lie in $\{0, M-1\}$, and the digit condition puts $d_{j,i}$ there as well. Write $d_{j,i} = (M-1)\varepsilon_{j,i}$ with $\varepsilon_{j,i} \in \{0,1\}$, and put $e_i := \sum_{j=1}^{m} 2^{j-1} \varepsilon_{j,i} \in \{0, 1, \ldots, M-1\}$. For $i \geq t$ the identity \eqref{eq:carry} reads
        \begin{equation}\label{eq:carry-stable}
            \sigma_{i-1} + (M-1) e_i \ =\ P_i + M \sigma_i .
        \end{equation}
    Reducing \eqref{eq:carry-stable} modulo $M$ gives $\sigma_{i-1} - e_i \equiv P_i$, while \eqref{eq:carry-bound} and $0\leq e_i\leq M-1$ give $-(M-1) \leq \sigma_{i-1}-e_i\leq M-2$. Only two cases are possible: $e_i = \sigma_{i-1}$ when $P_i = 0$, and $e_i = \sigma_{i-1} + 1$ when $P_i = M-1$. Substituting either identity for $e_i$ into \eqref{eq:carry-stable} yields $\sigma_i = \sigma_{i-1}$.

    The digit $P_i$ is constant for $i\geq t$. Since $\sigma_i=\sigma_{i-1}$ for all $i\geq t$, the carries are constant from position $t-1$ onward. By \eqref{eq:carry-stable}, $e_i$ is constant for all $i\geq t$. Since $e_i=\sum_{j=1}^m 2^{j-1}\varepsilon_{j,i}$ and binary expansions are unique, each $\varepsilon_{j,i}$ is constant for $i\geq t$. Hence the $M$-adic expansion of every $a_j$ is stable from position $t$ onward, and thus, $\abs{a_j}\leq M^t\leq MN$. Since also $\abs{b_j}\leq N$ for every $j$, $\norm{x}_\infty \leq MN = M\norm{F(x)}_\infty$.

    \emph{Sharpness.} For the upper bound, take $z=\bigl( (1,0),\ldots,(1,0)\bigr)$. For the lower bound, consider
$x = \bigl( (-M, -1), \ldots, (-M,-1),\ (M-2, 1) \bigr)$.
    Each pair lies in $X_m$, since $-M = [\overline{M-1},\, 0]_M$, $-1 = [\overline{M-1}]_M$, $M-2 = [\overline{0},\, M-2]_M$, and $1 = [\overline{0},\, 1]_M$. Then
$\sum_{j=1}^{m} 2^{j-1} a_j = -M(2^{m-1}-1) + 2^{m-1}(M-2) = 0$, so $\norm{F(x)}_\infty = 1$ while $\norm{x}_\infty = M$.
\end{proof}

\subsection{Assembling the space-filling surface} \label{subsec:surface}

Recall that $E_k$ denotes the $k$-th level approximation of $E_\infty$, defined in Section~\ref{sec:selfsimilar}.

\begin{figure}[t]
    \centering
    \includegraphics[width=.70\textwidth]{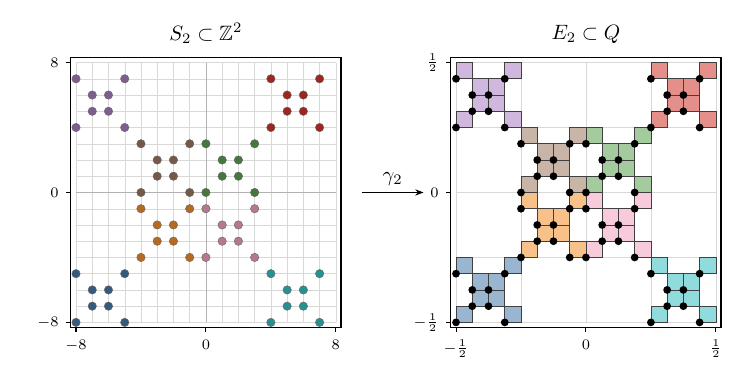}
    \caption{The dilation $\gamma_2$ carries $S_2$ from Stong's set onto the corners of the family of squares $E_2$. Illustrated with $m=2$.}
    \label{fig:lattice}
\end{figure}

\begin{lemma}\label{lem:gamma}
    For $k \geq 1$, let $\gamma_k \colon \RR^{2} \rightarrow \RR^{2}$ be the dilation $\gamma_k(x) := M^{-k}x$. Then $\gamma_k$ carries $S_k = X_m \cap [-h_k, h_k)^{2}$ bijectively onto the set of lower-left-hand corners of the $M^{-k}$-squares forming $E_k$; see Figure~\ref{fig:lattice}.
\end{lemma}

\begin{proof}
    We argue by induction on $k$. The lower-left-hand corners of the squares of $E_1$ are the points $\gamma_1(a,b)$ with $(a,b) \in [-h_1, h_1)^{2}$ and $d_0 = c_0$ or $d_0 = M-1-c_0$. These are exactly the $2M$ squares $Q_{p,q}$ with $(p,q) \in \mathcal{A}$, which settles the case $k=1$.

    Suppose the claim holds for some $k \geq 1$. The squares of $E_{k+1}$ are obtained by replacing each square $\phi_w(Q)$ of $E_k$ by the $2M$ squares $\phi_{w(d,c)}(Q)$ with $(d,c) \in \mathcal{A}$, whose lower-left-hand corner is that of $\phi_w(Q)$ displaced by $M^{-(k+1)}(d,c)$. On the lattice side, multiplying by $M$ moves every digit of $a$ and $b$ up one place and writes a $0$ in the lowest. The assignment $(a,b) \mapsto (Ma+d,\, Mb+c)$ therefore carries $S_k \times \mathcal{A}$ bijectively onto $S_{k+1}$: the digit condition holds in the higher places because it holds for $(a,b)$, and in the lowest because $(d,c) \in \mathcal{A}$. As $\gamma_{k+1}(Ma+d,\, Mb+c) = \gamma_k(a,b) + M^{-(k+1)}(d,c)$, the two agree.
\end{proof}

\begin{lemma}\label{lem:lattice-in-E}
    For every $k \geq 1$ we have $\gamma_k(S_k) \subset \gamma_{k+1}(S_{k+1}) \subset E_\infty$, and
$E_\infty = \overline{\bigcup_{k \geq 1} \gamma_k(S_k)}$.
\end{lemma}

\begin{proof}
        Let $v$ be the lower-left corner of $Q$. By Lemma~\ref{lem:gamma} the points of $\gamma_k(S_k)$ are exactly the points $\phi_w(v)$ with $\abs w = k$. Since $v \in E_\infty$ by \eqref{eq:corners} and $\phi_w(E_\infty) \subset E_\infty$, every such point lies in $E_\infty$. The index $(0,0)$ belongs to $\mathcal{A}$ and $v$ is the fixed point of $\phi_{0,0}$. It follows that $\phi_w(v) = \phi_{w(0,0)}(v)$, which exhibits each point of $\gamma_k(S_k)$ as a point of $\gamma_{k+1}(S_{k+1})$. Finally, any $x \in E_\infty$ lies in $\phi_w(Q)$ for some word $w$ of length $k$, and $\phi_w(v)$ is a corner of that square, whence $\norm{x - \phi_w(v)}_\infty \leq \diam \phi_w(Q) = M^{-k}$. The union is therefore dense in $E_\infty$, and it is contained in the compact set $E_\infty$.
\end{proof}

    Since $g$ maps $[0,1]$ onto $E_\infty$ and $L$ is defined on all of $\RR^{2m}$, the assignment
$(u_1, \ldots, u_m) \mapsto L\bigl( g(u_1), \ldots, g(u_m) \bigr)$ is defined on all of $[0,1]^{m}$ and has image $L\bigl(E_\infty^{m}\bigr)$. We now show that this image contains a cube centered at the origin.

\begin{lemma}\label{lem:cube}
    Put $B := \bigl[ -\tfrac{1}{2M},\ \tfrac{1}{2M} \bigr]^{m+1}$. Then $B \subset L\bigl(E_\infty^{m}\bigr)$.
\end{lemma}

\begin{proof}
    The set $E_\infty^{m}$ is compact and $L$ is continuous. We see that $L\bigl(E_\infty^{m}\bigr)$ is compact, hence closed. Because the dyadic rationals $\ZZ[1/2] = \bigcup_{k \geq 1} M^{-k} \ZZ$ are dense in $\RR$, to show $B \subset L\bigl(E_\infty^{m}\bigr)$, it will suffice to prove that every $y \in \ZZ[1/2]^{m+1}$ with $\norm{y}_\infty < \tfrac{1}{2M}$ belongs to $L\bigl(E_\infty^{m}\bigr)$.

    Fix such a $y$ and choose $k \geq 1$ with $z := M^{k} y \in \ZZ^{m+1}$. Then
$\norm{z}_\infty = M^{k}\norm{y}_\infty < M^{k}/(2M) = h_k/M$.
    Let $x := F^{-1}(z)$, which is defined by Proposition~\ref{prop:bij}. The lower bound of Proposition~\ref{prop:bilip} gives $\norm{x}_\infty \leq M \norm{z}_\infty < h_k$. Since that inequality is strict, every coordinate pair of $x$ lies in $X_m \cap (-h_k, h_k)^{2} \subset S_k$. Hence $\gamma_k(x) \in \bigl( \gamma_k(S_k) \bigr)^{m} \subset E_\infty^{m}$ by Lemma~\ref{lem:lattice-in-E}. Since $\gamma_k$ is multiplication by the scalar $M^{-k}$ and $L$ is linear,
$L\bigl( \gamma_k(x) \bigr) = M^{-k} L(x) = M^{-k} F(x) = M^{-k} z = y$.
\end{proof}

\begin{proof}[Proof of Theorem~\ref{thm:main}]
    Let $\pi$ denote the nearest-point projection of $\RR^{m+1}$ onto $B$, which is coordinatewise truncation. Hence $\pi$ is $1$-Lipschitz for $\norm{\cdot}_\infty$ and the identity on $B$. Let $T(y) := My + \tfrac12(1, \ldots, 1)$, which carries $B$ bijectively onto $[0,1]^{m+1}$ and scales $\norm{\cdot}_\infty$ by $M$. Write $\mathbf{g} := \bigotimes_1^m g$, so that $\mathbf{g}(u) = \bigl( g(u_1), \ldots, g(u_m) \bigr) \in \RR^{2m}$ for $u \in [0,1]^{m}$ and define $f := T \circ \pi \circ L \circ \mathbf{g}$.

    For the H\"older estimate, let $u, v \in [0,1]^{m}$. Then
        \begin{align*}
            \norm{f(u)-f(v)}_\infty
                &= M \norm{ \pi L \mathbf{g}(u) - \pi L \mathbf{g}(v) }_\infty
                \ \leq\ M \norm{ L\mathbf{g}(u) - L\mathbf{g}(v) }_\infty \\
                &\leq M(M-1) \norm{\mathbf{g}(u)-\mathbf{g}(v)}_\infty
                \ =\ M(M-1) \max_{j} \norm{g(u_j)-g(v_j)}_\infty \\
                &\leq 4M(M-1) \norm{u-v}_\infty^{\alpha}
                \ =\ 2^{m+2}(2^{m}-1) \norm{u-v}_\infty^{\alpha} ,
        \end{align*}
    by the scaling of $T$, the $1$-Lipschitz property of $\pi$, the linearity of $L$ together with the upper bound of Proposition~\ref{prop:bilip}, the supremum norm on $\RR^{2m}$, and the H\"older condition for $g$ from Proposition~\ref{thm:curve}. Hence $f$ satisfies \eqref{the-param}.

    For surjectivity, the image of $[0,1]^{m}$ under $L \circ \mathbf{g}$ is $L\bigl(E_\infty^{m}\bigr)$, which contains $B$ by Lemma~\ref{lem:cube}. Since $\pi$ is the identity on $B$, the composite $\pi \circ L \circ \mathbf{g}$ maps $[0,1]^{m}$ onto $B$, and $T$ carries $B$ onto $[0,1]^{m+1}$.
\end{proof}

\section{Corollaries and H\"older constants} \label{sec:further}

\subsection{Euclidean spaces and convex bodies} \label{subsec:consequences}

\begin{proof}[Proof of Corollary~\ref{cor:global}]
    Let $1 \leq m < n$. A composition of an $\alpha$-H\"older map with a $\beta$-H\"older map is $\alpha\beta$-H\"older. Composing the surjections of Theorem~\ref{thm:main} in dimensions $m, m+1, \ldots, n-1$ gives a surjection of $[0,1]^{m}$ onto $[0,1]^{n}$ of exponent
$\tfrac{m}{m+1} \cdot \tfrac{m+1}{m+2} \cdots \tfrac{n-1}{n} = \tfrac{m}{n}$.
    When $m = 1$ the first map in the chain is instead a $(1/2)$-H\"older space-filling curve. The case $m = n$ is the identity.
\end{proof}

Two more results follow with little extra work. In the first, the cubes are replaced by Euclidean spaces, and in the second, by convex bodies.

\begin{cor}\label{cor:euclidean}
    For all $1 \leq m \leq n$, there exist $(m/n)$-H\"older continuous maps from $\RR^{m}$ onto $\RR^{n}$.
\end{cor}

\begin{proof}
    Let $f:[0,1]^m\rightarrow[0,1]^n$ be an $\alpha$-H\"older surjection, $\alpha:=m/n$. If $Q \subset \RR^{m}$ is a cube of side length $s$ and $R \subset \RR^{n}$ is a cube of side length $s^{\alpha}$, then composing $f$ with the dilations and translations carrying $Q$ onto $[0,1]^{m}$ and $[0,1]^{n}$ onto $R$ gives an $\alpha$-H\"older surjection from $Q$ onto $R$ with the same H\"older constant as $f$. Take such surjections $F_k$ from cubes $Q_k \subset \RR^{m}$ onto the cubes $[-k,k]^{n}$, and place the cubes $Q_k$ so far apart that points in distinct cubes are separated, in the snowflaked distance $\norm{x-y}_\infty^{\alpha}$, by more than the diameter of the larger of the two image cubes.
    \par\vspace{5pt}
\centerline{%
      \definecolor{cubeone}{HTML}{4878A8}%
      \definecolor{cubetwo}{HTML}{F28E2B}%
      \definecolor{cubethree}{HTML}{59A14F}%
      \begin{tikzpicture}[baseline=(current bounding box.center),x=1cm,y=1cm,
          >=stealth,
          ipbox/.style   = {draw=#1!85!black, fill=#1!18, line width=0.4pt},
          ipbox/.default = black,
          iprule/.style  = {black!25, line width=0.3pt},
          iplab/.style   = {font=\scriptsize}]
        \draw[iprule] (-0.35,0) -- (4.05,0);
        \draw[ipbox=cubeone]   (0,0)    rectangle (0.34,0.34);
        \draw[ipbox=cubetwo]   (1.15,0) rectangle (1.67,0.52);
        \draw[ipbox=cubethree] (2.90,0) rectangle (3.66,0.76);
        \node[iplab,anchor=north] at (0.17,-0.07) {$Q_1$};
        \node[iplab,anchor=north] at (1.41,-0.07) {$Q_2$};
        \node[iplab,anchor=north] at (2.3,0.5) {$\cdots$};
        \node[iplab,anchor=north] at (3.28,-0.07) {$Q_k$};
        \draw[->,line width=0.5pt] (4.35,0.38) -- (5.15,0.38)
          node[midway,above,iplab] {$F_k$};
        \draw[ipbox=cubethree,fill=none] (5.47,-0.20) rectangle (6.63,0.96);
        \draw[ipbox=cubetwo,  fill=none] (5.66,-0.01) rectangle (6.44,0.77);
        \draw[ipbox=cubeone,  fill=none] (5.85,0.18)  rectangle (6.25,0.58);
        \node[iplab,anchor=west] at (6.80,0.38) {$[-k,k]^{n}$};
      \end{tikzpicture}}
    \par\vspace{3pt}\noindent
    The map that agrees with $F_k$ on each $Q_k$ is then $\alpha$-H\"older on the union of the cubes, because a pair of points in distinct cubes satisfies the H\"older condition with constant 1. Since the $\alpha$-H\"older maps are precisely the maps that are Lipschitz for the snowflaked distance, the McShane extension theorem, applied to each coordinate, extends it to an $\alpha$-H\"older map on $\RR^{m}$, see e.g.~\cite{Hei01}. As the cubes $[-k,k]^{n}$ exhaust $\RR^{n}$, the extension is surjective.
\end{proof}

\begin{cor}\label{cor:convex}
    Let $1 \leq m \leq n$. If $J \subset \RR^{m}$ has nonempty interior and $K \subset \RR^{n}$ is a nonempty compact convex set, then there exists an $(m/n)$-H\"older continuous surjection from $J$ onto $K$.
\end{cor}

\begin{proof}
    Since $J$ has nonempty interior, it contains a closed cube $Q$, and $K$,
    being bounded, is contained in a closed cube $R$. Composing the surjection
    of Corollary~\ref{cor:global} with affine identifications of $Q$ and $R$
    with unit cubes gives an $(m/n)$-H\"older surjection $h$ of $Q$ onto $R$.
    All norms on a finite-dimensional vector space are bi-Lipschitz
    equivalent, and composing a H\"older map with Lipschitz maps leaves the
    exponent unchanged. Thus, we may use the Euclidean norm. Let $p$ be the
    nearest-point projection of $\RR^{m}$ onto $Q$ and let $\pi$ be the
    nearest-point projection of $\RR^{n}$ onto $K$. Both are $1$-Lipschitz,
    see e.g.~\cite[Proposition 5.3]{Brezis2011}, and each is the identity on
    its image. Since $Q \subset J$ and $K \subset R$, the restriction of $\pi
    \circ h \circ p$ to $J$ is an $(m/n)$-H\"older surjection of $J$ onto $K$.
\end{proof}

\subsection{H\"older constants} \label{sec:best-constant}

For a $(1/2)$-H\"older space-filling curve $\gamma:[0,1]\rightarrow[0,1]^2$, the square of the H\"older constant with respect to the Euclidean norm on $\RR^{2}$, that is, $\kappa(\gamma):= \sup_{t_1<t_2} \norm{\gamma(t_2)-\gamma(t_1)}_2^{2}/(t_2-t_1)$, is called the \emph{square-to-linear ratio}. The least possible value $\kappa=\inf_{\gamma}\kappa(\gamma)$ satisfies $3.625\leq \kappa\leq 4$. The upper bound is attained by (scaling limits of) the Niedermeier-Reinhardt-Sanders H-curve \cite{NRS02}, and the lower bound is proved by Shchepin and Mychka~\cite{SM21}. Malykhin and Shchepin~\cite{MS23} make an extensive study of $\Hold_{1/n}(\gamma)^n$ for space-filling curves $\gamma:[0,1]\rightarrow[0,1]^n$ when $n\geq 2$.

Cube-to-square ratios could be defined similarly for maps between cubes of higher dimension, but we prefer to think directly about H\"older constants. The volume estimate \eqref{vol-est} with $\alpha = m/(m+1)$ gives $H \geq 1$ for every $H$ admissible in \eqref{the-param}, and the following lemma improves that bound.

\begin{lemma}\label{thm:main2} If $f:[0,1]^m\rightarrow[0,1]^{m+1}$ is surjective and satisfies \begin{equation}\|f(x)-f(y)\|_\infty \leq H \|x-y\|^{m/(m+1)}_\infty\quad\text{for all }x,y\in[0,1]^m,\end{equation} then $H\geq 2^{m/(m+1)}$.\end{lemma}

\begin{proof}
    The $2^{m+1}$ vertices of $[0,1]^{m+1}$ are pairwise at distance $1$ in the supremum norm. Since $f$ is onto we may choose a preimage of each vertex, and distinct vertices have distinct preimages. This gives $2^{m+1}$ distinct points of $[0,1]^{m}$. Cut $[0,1]^{m}$ into the $2^{m}$ closed cubes of side $\tfrac12$ that meet at the center. Since $2^{m+1} > 2^{m}$, two of the chosen preimages $x$ and $y$ lie in a common cube. Then $\norm{x-y}_\infty \leq \tfrac12$ while $\norm{f(x)-f(y)}_\infty = 1$. The H\"older condition now gives
$1 = \norm{f(x)-f(y)}_\infty \leq H \norm{x-y}_\infty^{m/(m+1)} \leq H \, 2^{-m/(m+1)}$,
    which is the assertion.
\end{proof}

The least H\"older constant of a surjection from $[0,1]^{m}$ onto $[0,1]^{m+1}$ lies between $2^{m/(m+1)}$ and $2^{m+2}(2^{m}-1)$.

\begin{question}\label{q:constant}
    Do there exist $m/(m+1)$-H\"older maps from $[0,1]^{m}$ onto $[0,1]^{m+1}$ whose H\"older constants are bounded by a universal constant (independent of $m$)?
\end{question}

\providecommand{\bysame}{\leavevmode\hbox to3em{\hrulefill}\thinspace}
\providecommand{\MR}{\relax\ifhmode\unskip\space\fi MR }
\providecommand{\MRhref}[2]{%
  \href{http://www.ams.org/mathscinet-getitem?mr=#1}{#2}
}
\providecommand{\href}[2]{#2}

\end{document}